\documentclass[12pt]{article}

\usepackage{setspace}
\usepackage{latexsym}
\usepackage{amssymb}
\usepackage{amsthm}
\usepackage{amsfonts}
\usepackage{amsmath}
\usepackage{enumerate}
\usepackage{enumitem}
\usepackage{url, doi}
\hypersetup{citecolor=blue, linkcolor=blue, colorlinks=true}
\usepackage{longtable, float, booktabs, cancel, cite}
\usepackage[margin=1in,nohead]{geometry}
\usepackage{colonequals}
\usepackage{titlesec}
\usepackage{longtable}
\usepackage{comment}

\usepackage{tikz}

\usepackage{hyperref}

\usetikzlibrary{decorations.markings}
\usetikzlibrary{arrows}

\tikzstyle{vertex}=[circle,draw,inner sep=0pt, minimum size=6pt]

\makeatletter
\def\namedlabel#1#2{\begingroup
    #2%
    \def\@currentlabel{#2}%
    \phantomsection\label{#1}\endgroup
}
\makeatother

\newtheorem{thm}{Theorem}[section]
\newtheorem{cor}[thm]{Corollary}

\newtheorem{lem}[thm]{Lemma}

\newtheorem{prop}[thm]{Proposition}
\theoremstyle{definition}\newtheorem{rem}[thm]{Remark}
\newtheorem{ex}[thm]{Example}

\newcommand{\Aut}{{\rm Aut}}

\newcommand{\GAP}{\textbf{{\rm GAP}}}

\newcommand{\GL}{{\rm GL}}
\newcommand{\SO}{{\rm SO}}

\newcommand{\AGL}{{\rm AGL}}

\newcommand{\ASO}{{\rm ASO}}

\newcommand{\Cay}{{\rm Cay}}

\newcommand{\calP}{\mathcal{P}}
\newcommand{\calB}{\mathcal{B}}
\newcommand{\calA}{\mathcal{A}}
\newcommand{\calX}{\mathcal{X}}

\renewcommand\le{\leqslant}
\renewcommand\ge{\geqslant}

\newcommand{\F}{\mathbb{F}}

\def\Box{\vcenter{\vbox{\hrule\hbox{\vrule
     \vbox to 8.8pt{\hbox to 10pt{}\vfill}\vrule}\hrule}}}

\newcommand{\Z}{{\mathbb Z}}

\newcommand{\ep}{\varepsilon}

\newcommand{\VO}{{\rm VO}}
\newcommand{\VNO}{{\rm VNO}}

\begin{document}
 
\title{Nonabelian partial difference sets constructed using abelian techniques}
\author{
James A. Davis\footnote{University of Richmond, VA 23173, {\tt jdavis@richmond.edu}},
John Polhill\footnote{Commonwealth University, Bloomsburg, PA, {\tt jpolhill@commonwealthu.edu}},
Ken Smith\footnote{Huntsville, TX 77340, {\tt kenwsmith54@gmail.com}},
Eric Swartz\footnote{William \& Mary, Williamsburg, VA 23187, {\tt easwartz@wm.edu}}
}

\date{}

\maketitle

\begin{abstract}
A $(v,k,\lambda, \mu)$-partial difference set (PDS) is a subset $D$ of a group $G$ such that $|G| = v$, $|D| = k$, and every nonidentity element $x$ of $G$ can be written in either $\lambda$ or $\mu$ different ways as a product $gh^{-1}$, depending on whether or not $x$ is in $D$.  Assuming the identity is not in $D$ and $D$ is inverse-closed, the corresponding Cayley graph $\Cay(G,D)$ will be strongly regular.  Partial difference sets have been the subject of significant study, especially in abelian groups, but relatively little is known about PDSs in nonabelian groups.  While many techniques useful for abelian groups fail to translate to a nonabelian setting, the purpose of this paper is to show that examples and constructions using abelian groups can be modified to generate several examples in nonabelian groups.  In particular, in this paper we use such techniques to construct the first known examples of PDSs in nonabelian groups of order $q^{2m}$, where $q$ is a power of an odd prime $p$ and $m \ge 2$.  The groups constructed can have exponent as small as $p$ or as large as $p^r$ in a group of order $p^{2r}$.  Furthermore, we construct what we believe are the first known Paley-type PDSs in nonabelian groups and what we believe are the first examples of Paley-Hadamard difference sets in nonabelian groups, and, using analogues of product theorems for abelian groups, we obtain several examples of each.  We conclude the paper with several possible future research directions.  
\end{abstract}

\section{Introduction}

This work focuses on the algebraic structure known as a partial difference set (PDS).  A $(v, k, \lambda, \mu)$-PDS is a subset $D$ of a group $G$ such that $|G| = v$; $|D| = k$; every nonidentity element of $D$ can be written as $d_1d_2^{-1}$, where $d_1, d_2 \in D$, in $\lambda$ different ways; and every nonidentity element of $G - D$ can be written as $d_1d_2^{-1}$, where $d_1, d_2 \in D$, in $\mu$ different ways. These sets have received much attention due to their correspondences with strongly regular graphs, codes, bent functions, and association schemes.

Over the past few decades, numerous constructions of PDSs have been given in many {\it abelian} groups (for example, \cite{Davis_1994, Feng_Momihara_Xiang_2015, Leung_Ma_1990, Malandro_Smith_2020, Momihara_Xiang_2018, Polhill_2008, Polhill_2019}). The methods nearly always include the use of characters, no doubt because they provide a relatively simple proof. Recent work has shed light on the fact that PDSs can be constructed in nonabelian groups as well, see for instance \cite{Brady_2022, DeWinter_etal_2016, Feng_He_Chen_2020, Polhill_etal_2023, Swartz_2015, Swartz_Tauscheck_2021}. We believe that nonabelian groups will provide many interesting examples of PDSs, even though relatively few examples are known in this setting (see \cite[Sections 4--5]{Polhill_etal_2023} for a recent survey). 

In a previous paper \cite{Polhill_etal_2023}, the authors investigated PDSs in nonabelian groups for which there are no abelian PDSs with those parameters. In this situation, both the parameters and the PDS itself are called \textit{genuinely nonabelian}.  On the other end of the spectrum, there are examples of PDSs in nonabelian groups that are not genuinely nonabelian, such as in \cite{Feng_He_Chen_2020}; that is, given a $(v, k, \lambda, \mu)$-PDS in an abelian group, there also exists a $(v, k, \lambda, \mu)$-PDS in a nonabelian group.  While one of the main themes of \cite{Polhill_etal_2023} is that many tools from the abelian setting simply do not apply to nonabelian groups, the purpose of this paper is to show that several constructions and existence results in abelian groups have direct analogues in nonabelian groups.  In fact, in several instances, both the abelian group and the nonabelian group act on the same underlying combinatorial object (in this case, the same strongly regular graph).

The main results of this paper can be summarized as follows. 
\begin{itemize}
 \item[(1)] Let $q$ be a power of an odd prime $p$ and $m \ge 2$.  Then, there exist nonisomorphic, nonabelian groups of order $q^{2m}$ and exponent $p$ whose nonidentity elements can be partitioned into PDSs  (Theorem \ref{thm:nonabpds}, Remark \ref{rem:G1expp}, Theorem \ref{thm:nonabpds2}, Remark \ref{rem:G2expp}, Theorem \ref{thm:G1notG2}).
 \item[(2)] Let $q$ be a power of an odd prime $p$.  There exists a nonabelian group of order $q^4$ and exponent $p$ that can be partitioned into $q + 3$ PDSs, the union of any which is also a PDS (Theorem \ref{thm:affassocscheme}).  In particular, this group contains a Paley-type PDS (Corollary \ref{cor:affpaley}).
 \item[(3)] Let $t \ge 2$ and $p$ be an odd prime.  The group 
 \[ \widehat{G_t} \colonequals \left\langle x, y : x^{p^t} = y^{p^t} = 1, yxy^{-1} = x^{(p-1)p^{t-1} + 1} \right\rangle \cong \Z_{p^t} \rtimes_{(p-1)p^{t-1} + 1} \Z_{p^t}\]
 can be partitioned into $2p$ PDSs in such a way that the union of any of them is also a PDS.  In particular, there is a Paley-type PDS in $\widehat{G_t}$ (Theorem \ref{mainPaleyconstr}).  
 \item[(4)]  If two groups $G$ and $G'$ of order $v$ possess Paley-type PDSs, then the group $G \times G'$ also contains a Paley-type PDS  (Theorem \ref{generalproduct}). Combined with the results of (2) and (3), this provides infinitely many more examples of Paley-type PDSs in nonabelian groups.  
 \item[(5)] If a group $G$ of order $v$ contains a Paley-type PDS and the group $G'$ of order $v \pm 2$ contains a skew Hadamard difference set (DS), then the product group $G \times G'$ contains a Paley-Hadamard DS in the Stanton-Sprott (Twin prime power) family (Theorem \ref{twinprimes}).  Combined with the results of (2) and (3), this provides more examples of Paley-Hadamard difference sets in nonabelian groups. 
 \item[(6)] In many cases, the group ring calculations needed to prove a product theorem in the abelian case (that is, the existence of PDSs in abelian groups $G$ and $G'$ imply the existence of a PDS in $G \times G'$) do not depend on whether or not the groups are abelian, meaning that they will automatically translate to the nonabelian setting (Lemma \ref{GroupRingRelationsLemma}, Theorem \ref{Productstononabelian}).
\end{itemize}

We believe the PDSs constructed in this paper are the first infinite families of PDSs in nonabelian groups of order $q^d$, where $q$ is an odd prime power and $d > 3$. (Partial difference sets of have been constructed in nonabelian groups of order $q^3$, where $q$ is an odd prime power, in \cite{Swartz_2015} and \cite{Polhill_etal_2023}.)  Moreover, in this paper we construct what we believe are the first known Paley-type PDSs in nonabelian groups and what we believe are the first examples of Paley-Hadamard DSs in nonabelian groups. 

This paper is oraganized as follows.  Section \ref{sect:prelim} contains preliminary material related to PDSs, association schemes, and quadratic forms. In Section \ref{sect:aff}, we will use geometric techniques to construct families of PDSs in certain nonabelian groups or order $q^{2m}$ and exponent $p$, where $p$ is an odd prime and $q$ is a power of $p$. In Section \ref{sec:SemidirectLargeCenter}, we will use group ring equations from the abelian world to obtain PDSs in groups of the form $G = \Z_{p^r} \rtimes \Z_{p^r}$ that have a large center, $Z(G) = pG \cong \Z_{p^{r-1}} \times \Z_{p^{r-1}}$. Section \ref{sec:Paley} highlights new constructions of Paley-type PDSs and how they can be used to construct new difference sets using the twin prime power construction. In Section \ref{sec:product}, we show that group ring equations will be forced to hold for various product constructions previously shown to work for abelian groups by characters. In this case, since the input group ring relations for the product are identical for the abelian case as for the nonabelian case we can avoid messy group ring equations. The product will take as input the nonabelian PDSs from Sections \ref{sect:aff}  and \ref{sec:SemidirectLargeCenter} and will generate new PDSs in many nonabelian groups. We conclude in Section \ref{sect:possnext} with some remarks and a considerable list of open problems.


\section{Preliminaries}
\label{sect:prelim}

\subsection{Partial difference sets}

Let $G$ be a finite group of order $v$ with a subset 
$D$ of order $k$.  Suppose further that the differences $d_1{d_2}^{-1}$ for
$d_1, d_2 \in D, d_1 \neq d_2$, represent each 
nonidentity element of $G$ precisely $\lambda$ times. Then, $D$ is a {\it $(v,k, \lambda)$-difference set} (DS) in $G$.

Now suppose that $G$ is a finite group of order $v$ with a subset 
$D$ of order $k$.  Suppose further that the differences $d_1{d_2}^{-1}$ for
$d_1, d_2 \in D, d_1 \neq d_2$, represent each 
nonidentity element of $D$ exactly $\lambda $ times and the  
nonidentity element of $G - D$ exactly $\mu $ times.  Then, $D$ is called a 
{\it $(v, k , \lambda, \mu)$-partial difference set} (PDS) in $G$.  The
survey article of Ma is an excellent resource for these sets \cite{Ma_1994}.  Typically, a proper PDS $D$
for which $\lambda \neq \mu$ will have the two properties that the identity element from $G$ is not in $D$ and that $x \in D$ implies 
$x^{-1} \in D$, and such a PDS is called {\it regular}.
A PDS having parameters $(n^2, r(n-1),
n+r^2-3r,r^2-r)$ for some natural number $r$ is called a {\it Latin square type PDS}. 
Similarly, a PDS having parameters $(n^2, r(n+1),
-n+r^2+3r,r^2+r)$ is called a {\it negative Latin square type PDS}.
Assuming the PDS is regular, the Cayley graph for a $(v,k,\lambda , \mu)$-PDS will always be a \textit{strongly regular graph} with the same parameters; that is, the corresponding Cayley graph has $v$ vertices, every vertex has $k$ neighbors, adjacent vertices have $\lambda$ common neighbors, and nonadjacent vertices have $\mu$ common neighbors.  

The earliest examples of PDSs date back to Paley \cite{Paley_1933}, though his work long predates the systematic study of PDSs.  Paley showed that the 
nonzero squares in $\F_q$ will be a $(q, \frac{q-1}{2},\frac{q-5}{4},\frac{q-1}{4})$-PDS in the additive group when $q$ is a prime power and $q \equiv 1 \pmod 4$.  We will call these {\it Paley partial difference sets}, and more generally $(v, \frac{v-1}{2},\frac{v-5}{4},\frac{v-1}{4})$-PDSs will be {\it Paley-type partial difference sets}.  This family of PDSs has received much attention in abelian groups.  Davis was the first to construct Paley-type PDSs in groups that are not elementary abelian \cite{Davis_1994}, work that was subsequently generalized in \cite{Leung_Ma_1995} and \cite{Polhill_2002}.  Polhill found examples where $v$ was not a prime power in \cite{Polhill_2009b}.  

Paley \cite{Paley_1933} showed in the case when $q$ is a prime power and $q \equiv 3 \pmod 4$ that the set of nonzero squares will instead be a $(q,\frac{q-1}{2},\frac{q-3}{4})$-difference set, now called a {\it Paley-Hadamard difference set}.  Stanton and Sprott found new examples of Paley-Hadamard difference sets \cite{Stanton_Sprott_1958} which are known as {\it Twin prime power difference sets} in the additive group of $\F_q \times \F_{q+2}$, when $q$ and $q+2$ are both prime powers. 

Partial difference sets are often studied within the context of the particular group ring $\Z[G]$, whether the group $G$ is abelian or not.  For a subset $D$ of a 
group $G$, we abuse notation slightly and write $D \colonequals \sum_{d \in D} d$ and $D^{(-1)} \colonequals \sum_{d \in D}d^{-1}$.   The following equation will then hold for a regular $(v,k,\lambda, \mu)$-partial
difference set $D$ in the group $G$ with identity $1_G$:

$$ DD^{(-1)} = DD= \lambda D + \mu (G - D - 1_G) + k 1_G = (\lambda - \mu)D + \mu G + (k - \mu) 1_G. $$

\subsection{Association schemes}

When studying PDSs, and in particular those with (negative) Latin square type parameters, one often has a partition of the nonidentity elements into multiple PDSs. As such, they form a multi-class association scheme, and so it will be helpful to consider these mathematical structures.

Let $\calX$ be a finite set. An {\it association scheme} with $t$ classes on $\calX$ is a partition of
$\calX\times \calX$ into sets $R_0$, $R_1, \ldots , R_t$ (relations, or associate classes) such that

\begin{enumerate}
\item $R_0=\{(x,x) : x\in \calX\}$ (the diagonal relation);
\item for each $\ell$, $R_{\ell}^t = \{ (y,x) : (x,y) \in R_{\ell} \} = R_{\ell'}$ for some $\ell'$;
\item for all $i,j,k$ in $\{0,1,2,\ldots ,t\}$ there is an integer $p_{ij}^k$ such that, for all $(x,y)\in R_k$,
$$|\{z\in \calX : (x,z)\in R_i\; {\rm and}\; (z,y)\in R_j\}|=p_{ij}^k.$$
\end{enumerate}

When $p_{ij}^k = p_{ji}^k$ for all $k,i,j$ then the association scheme is called {\it commutative}.
If $\ell = \ell'$ for all $\ell$, then the association scheme is said to be {\it symmetric}; 
otherwise, it is {\it nonsymmetric}.   

Each of the relations $R_l$ can be interpreted as a directed
graph with vertex set $\calX$ and edge set $R_l$, $\Gamma_l = (\calX,R_l)$ for all $l$.
An association scheme can be viewed as a decomposition of the 
complete directed graph with vertex set $\calX$ into directed graphs $\Gamma_l$ with the 
property that for $i,j,k \in \{ 1,2,\cdots d \}$ and for $xy \in 
E(\Gamma_k)$,
$$|\{ z \in X : xz \in E(\Gamma_i) \text{ and } zy \in E(\Gamma_j) \}| = 
p^k_{ij},$$
where $E(\Gamma_i)$ is edge set of graph $\Gamma_i$.  The graphs $\Gamma_i$ are
called the {\it graphs} of the association scheme.  Likewise, a symmetric association scheme can be viewed as a decomposition of the complete graph on vertex set $\calX$ into undirected graphs.  A strongly regular
graph $\Gamma$ corresponds to a symmetric association scheme with two classes, where $R_1 = \{(x,y) : xy \in E(\Gamma)\}$ and $R_2 = \{(x,y) : x \neq y \text{ and } (x,y) \notin R_1 \}$.  

For an association scheme, we can interpret the relationas as adjacency matrices for the graphs, i.e., $ \{ 0,1 \}$-matrices indexed by the vertex set $V$ such that for the matrix $A_i$ there is a 1 in position $ (x,y)$ exactly when $(x,y) \in R_i$.  Then,we have:
\

\begin{enumerate}
\item $A_0 = I$;
\item $A_0 + A_1 + ...+ A_d  = J,$ the matrix with all 1's;
\item for each $i$ there is some $i'$ with ${A_i}^T = A_{i'}$;
\item $A_iA_j = \Sigma_k {p_{ij}}^k A_k$.
\end{enumerate}

This collection forms what is known as the {\it Bose-Mesner algebra}, and what is key for this article is that for a commutative association scheme
it will necessarily follow that the Bose-Mesner algebra is commutative, so that the graph adjacency matrices satisfy $A_iA_j=A_jA_i$ for all $i,j$.

Given an association scheme, we can take unions of classes to produce graphs
with larger edge sets, with such unions termed {\it fusions}.  Fusions are not always
association schemes in general, but when a particular association scheme
has the property that any of its fusions also forms an association scheme
we call the scheme {\it amorphic}.  For an excellent introduction to amorphic association schemes, see \cite{vanDam_Muzychuk_2010}.

Partial difference sets give rise to strongly regular Cayley graphs.  When
we partition the nonidentity elements of a group into partial difference
sets, we also have a partitioning of the complete graph 
with vertex set the group elements into strongly regular Cayley graphs.  

Now we are ready to consider what will be an essential ingredient for many of the constructions in this article, a powerful result of van Dam:

\begin{thm}\cite[Theorem 3]{vanDam_2003}
\label{VDtheorem}
Let $\{ \Gamma_1,\Gamma_2,\dots, \Gamma_d \}$ be an edge-decomposition of the complete 
graph on a set $X$, where each $\Gamma_i$ is strongly regular.  If the $\Gamma_i$
are all of Latin square type or all of negative Latin square type,
then the decomposition is a $d$-class amorphic association scheme on $X$.
\end{thm}

We interpret the implications of this result into the context of PDSs to form the following, which we will use throughout the paper.  

\begin{cor}
\label{AmorphicPDS}
Suppose the nonidentity elements of a group $G$ can be partitioned into a collection of PDSs all of Latin square type or all of negative Latin square type, $\{ P_1, P_2,\dots,P_n \}$. Then, a union of any number of these PDSs is also a PDS of that same type.  Moreover, $P_iP_j = P_jP_i$ in the group ring $\Z[G]$.
\end{cor}

\begin{proof}
Such a collection of PDSs corresponds to a strongly regular Cayley graph decomposition of the complete graph on $|G|$ points, which will be amorphic by Theorem \ref{VDtheorem}.  As such, any fusion of the graph is a strongly regular graph of the same type as the PDSs $P_i$ and therefore any union of PDSs $\bigcup_i P_i$ corresponds to another PDS of that type.  Amorphic association schemes are commutative, and it follows that the graph adjacency matrices commute and therefore so do the group ring equations for the PDSs: i.e.,  $P_iP_j = P_jP_i$.
\end{proof}

We remark that such a partition of the nonidentity elements of a group $G$ is called a \textit{Cayley (association) scheme}.   Cayley schemes are equivalent to \textit{Schur rings} \cite{Klin_1985}, and amorphic association schemes of (negative) Latin square type were previously used in \cite{Feng_He_Chen_2020} to construct examples of PDSs in nonabelian $2$-groups.

\subsection{Quadratic forms}

Quadratic forms have been used for constructing PDSs of both Latin square type and negative Latin
square type (see~\cite{Ma_1994}).  Let $q$ be a power of a prime. We denote the field with $q$
elements by $\F_q$. A \textit{quadratic form} $Q$ on a $d$-dimensional vector space $\F_q^d$ over $\F_q$ is a function $Q:\F_q^d\to\F_q$ such that:
\begin{itemize}
\item[(i)] $Q(\alpha x)=\alpha^2 Q(x)$ for all $\alpha\in\F_q$ and all $x\in\F_q^d$, and
\item[(ii)] the function $\beta:\F_q^d\times\F_q^d\to\F_q$ given by $\beta(x,y)=Q(x+y)-Q(x)-Q(y)$ is $\F_q$-bilinear.
\end{itemize}
A quadratic form $Q$ is said to be \textit{nondegenerate} if $\beta(x,y)=0$ for all $y\in\F_q^d$ implies $x=0$.  We have the following well-known result.

\begin{thm}\cite[Theorem 3.28]{Ball_2015}
\label{thm:Qform}
 Let $Q$ be a nondegenerate quadratic form on $V = \F_q^{2m}$.  There exists a basis for $V$ such that exactly one of the following holds for all $x = (x_1, \dots, x_{2m}) \in V$:
 \begin{itemize}
  \item[(i)] $Q(x) = x_1x_2 + x_3x_4 + \cdots + x_{2m-1}x_{2m}$, or
  \item[(ii)] $Q(x) = x_1x_2 + x_3x_4 + \cdots + x_{2m-3}x_{2m-2} + x_{2m-1}^2 + bx_{2m}^2,$
where $-b$ is a nonsquare in $\F_q$.
 \end{itemize}
 \end{thm}

If (i) of Theorem \ref{thm:Qform} holds, then we say $Q$ is \textit{hyperbolic} and has type $\ep = +1$ (often denoted simply with ``$+$'' when used as a superscript), and, if (ii) of Theorem \ref{thm:Qform} holds, then we say $Q$ is \textit{elliptic} and has type $\ep = -1$ (often denoted simply with ``$-$'').

 
\section{Nonabelian PDS families related to affine polar graphs}
\label{sect:aff}

Let $q$ be an odd prime power and $m \ge 2$.  Let $V = \F_q^{2m}$ be equipped with a nondegenerate quadratic form $Q$ of type $\ep = \pm 1$.  In particular, by Theorem \ref{thm:Qform}, if $x = (x_1, \dots, x_{2m}) \in V$, we will assume
\[Q(x) = x_1x_2 + x_3x_4 + \cdots + x_{2m-1}x_{2m}\]
if $\ep = 1$, and we will assume
\[Q(x) = x_1x_2 + x_3x_4 + \cdots + x_{2m-3}x_{2m-2} + x_{2m-1}^2 + bx_{2m}^2, \]
where $-b$ is a nonsquare in $\F_q$, if $\ep = -1$.  Note that there is a nondegenerate symmetric bilinear form $\beta$ associated with $Q$.  

The graphs $\VO^\ep(2m, q)$ are defined by taking the vectors in $V$ to be vertices, with distinct vectors $u, v \in V$ adjacent if $Q(v - u) = 0$.  As noted in \cite[Section 3.3.1]{Brouwer_VanMaldeghem_2022}, $\VO^\ep(2m, q)$ is a strongly regular graph with
\begin{align*}
 v       &= q^{2m},\\
 k       &= (q^m - \ep)(q^{m-1} + \ep),\\ 
 \lambda & = q(q^{m-1} - \ep)(q^{m-2} + \ep) + q - 2,\\ 
 \mu     &= q^{m-1}(q^{m-1} + \ep).
\end{align*}

The graphs $\VNO^\ep(2m, q)$ are defined by taking the vectors in $V$ to be vertices, with distinct vectors $u, v \in V$ adjacent when $Q(u - v)$ is a nonzero square in $\F_q$.  As noted in \cite[Section 3.3.2]{Brouwer_VanMaldeghem_2022}, the graph $\VNO^\ep(2m, q)$ is a strongly regular graph with
\begin{align*}
 v       &= q^{2m},\\ 
 k       &= \frac{1}{2}(q-1)(q^m - \ep)q^{m-1},\\ 
 \lambda &= \frac{1}{4}q^{m-1}(q-1)(q^m - q^{m-1} - 2\ep) + \ep q^{m-1},\\
 \mu     &= \frac{1}{4}q^{m-1}(q-1)(q^m - q^{m-1} - 2\ep).
\end{align*}

Finally, we note that the complement to $\VO^\ep(2m, q) \cup \VNO^\ep(2m, q)$ in the complete graph on $V$ will itself be a strongly regular graph isomorphic to $\VNO^\ep(2m, q)$. To see this, note that this new graph has adjacency defined by $u \sim v$ when $Q(v - u)$ is a nonsquare in $\F_q$.  If $a$ is a nonsquare in $\F_q$, the map $\phi: v \mapsto av$ interchanges nonsquares with nonzero squares, and $Q(v^\phi - u^\phi) = Q(av - au) = a^2Q(v-u)$ is still a nonzero square if and only if $Q(v - u)$ is, meaning $\phi$ is an isomorphism between $\VNO^\ep(2m, q)$ and this new complement graph, which we will denote by $\VNO_2^\ep(2m, q)$.  Therefore, the complete graph on $V$ can be partitioned into $\VO^\ep(2m, q)$, $\VNO^\ep(2m, q)$, and $\VNO_2^\ep(2m,q)$.  We remark that when $\ep = +1$ the graphs are of Latin square type, and when $\ep = -1$ the graphs are of negative Latin square type.

\subsection{Automorphisms of affine polar graphs}
\label{sect:autaff}

We represent the elements of the affine general linear group $\AGL(2m,q)$ in the form $[M, u]$, where $M \in \GL(2m, q)$ and $u \in V$, where for all (row vectors) $v \in V,$
\[v^{[M,u]} \colonequals vM + u,\]
and multiplication in $\AGL(2m,q)$ is defined by  \[[M_1, v_1][M_2, v_2] = [M_1M_2, v_1M_2 + v_2].\]  
The special orthogonal group $\SO^\ep(2m,q)$ is the set of all determinant $1$ matrices in $\GL(2m,q)$ preserving the bilinear form $\beta$ (and quadratic form $Q$), and, given a subspace $U$ of $V$, denote the subgroup of translations of $V$ by vectors in $U$ by $T_U$, i.e.,
\[ T_U \colonequals \{ [I, u] : u \in U\},\]
where $I$ is the $2m \times 2m$ identity matrix.  Hence,
\[ \ASO^\ep(2m, q) \colonequals \{[M,v] : M \in \SO^\ep(2m,q), v \in V \} \cong T_V:\SO^\ep(2m, q),\]
where $M \in  \SO^\ep(2m,q)$ is identified naturally with $[M,0] \in \AGL(2m, q)$ and the ``$:$'' denotes a semidirect product. 

\begin{lem}
 \label{lem:autaff}
The group $\ASO^\ep(2m,q)$ is a subgroup of automorphisms of each of $\VO^\ep(2m, q)$, $\VNO^\ep(2m, q)$, and $\VNO_2^\ep(2m, q)$. 
\end{lem}

\begin{proof}
 Let $[M,w] \in \ASO^\ep(2m,q)$ and $u, v \in V$.  Then,
 \[ Q(v^{[M,w]} - u^{[M,w]}) = Q((vM + w) - (uM + w)) = Q((v-u)M) = Q(v-u),\]
 and so $[M,w]$ preserves adjacency in all three graphs.
\end{proof}

We remark that $T_V \subseteq \ASO^\ep(2m,q)$ is an elementary abelian regular subgroup of automorphisms of each of $\VO^\ep(2m, q)$, $\VNO^\ep(2m, q)$, and $\VNO_2^\ep(2m, q)$, and so the corresponding decomposition of the nonidentity elements of $T_V$ into PDSs is an amorphic Cayley scheme.  In the following subsections, we will find ``twists'' of the group $T_V$ in $\ASO^\ep(2m,q)$ -- roughly speaking, replacing certain translations $[I, v]$ by elements of the form $[M, v]$, where $M \neq I$ -- to provide new examples of PDSs in nonabelian groups.

\subsection{A family of PDSs in nonabelian groups of order $q^{2m}$}

Fix $\ep = \pm 1$.  Let $v$ be a nonsingular vector in $V$, i.e., $Q(v) \neq 0$, so $\left\langle v \right\rangle$ is a nonsingular subspace.  The stabilizer of $\left\langle v \right\rangle$ in $\SO^\ep(2m,q)$ contains an elementary abelian group $H$ of order $q$; see, e.g., \cite[Sections 2.2.1, 8.2]{Bray_Holt_RoneyDougal_2013}, \cite[Tables 3.5 E, F and Proposition 4.1.6]{Kleidman_Liebeck_1990}, or \cite[Section 3.7.4]{Wilson_2009}. Note that $vM = v$ for all $M \in H$: we know that $0M = 0$, and that $vM = v$ follows from the Orbit-Stabilizer Theorem.

Since $v$ is nonsingular, we have $V = \left\langle v \right\rangle \oplus v^\perp$, where
\[ v^\perp \colonequals \{ u \in V : \beta(u,v) = 0\};\]
to see this, note that the map $\beta( -, v): V \to \F_q$ is a linear transformation with kernel $v^\perp$.  Since $vM = v$ for all $M \in H$ and $M$ preserves $\beta$, we have $v^\perp M = v^\perp$, i.e., $v^\perp$ is an $H$-invariant subspace.  

\begin{rem}
 Choosing a nonsingular vector $v$ is not strictly necessary for this construction: as long as the elementary abelian group $H$ stabilizes a decomposition $V = \left\langle v \right\rangle \oplus U$ for some complementary subspace $U$ to $\left\langle v \right\rangle$, the construction will work. 
\end{rem}

Since $H$ is an elementary abelian group of order $q$, $H$ is naturally isomorphic to $(\F_q,+)$.  For each $\alpha \in \F_q$, we will denote by $A_\alpha$ the corresponding element of $H$ under this natural isomorphism.  Define \[ \calA \colonequals \{[A_\alpha, \alpha v] : \alpha \in \F_q\}.\]  Since $v$ is fixed by right multiplication by elements of $H$, $\calA$ is an elementary abelian group of order $q$ that is itself naturally identified with $(\F_q, +)$.

Recall that $T_{v^\perp}$ is the set of elements of the form $[I,u]$ for $u \in v^\perp.$  Then, for $u \in v^\perp$, we have
\[ [A_\alpha, \alpha v]^{-1} [I, u] [A_\alpha, \alpha v] = [A_\alpha^{-1}, -\alpha v] [I,u] [A_\alpha, \alpha v] = [I, uA] \in T_{v^\perp},\]
so $\calA$ normalizes $T_{v^\perp}.$

Define
\[ G_1^\ep\colonequals \left\langle T_{v^\perp}, \calA \right\rangle = T_{v^\perp}:\calA.\]

\begin{thm}
 \label{thm:nonabpds}
 The group $G_1^\ep$ is a nonabelian group of order $q^{2m}$ in which the nonidentity elements can be partitioned into $D_0 \cup D_1 \cup D_2$, where each $D_i$ is a PDS, $\Cay(G_1^\ep, D_0) \cong \VO^\ep(2m, q)$, and $\Cay(G_1^\ep, D_1) \cong \Cay(G_1^\ep,D_2) \cong \VNO^\ep(2m, q)$.  
\end{thm}

\begin{proof}
 First, we have $|G_1^\ep| = q^{2m}$ since $G_1^\ep = T_{v^\perp}:\calA$, $|T_{v^\perp}| = |v^\perp| = q^{2m-1}$, and $|\calA| = q$.  Moreover, since $H$ acts faithfully on $V$ and fixes $\left\langle v \right\rangle$ pointwise, there exist $u \in v^\perp$ and $A \in H$ such that $uA \neq u$.  Since $A \in H$, there is a unique $w \in \left\langle v \right\rangle$ such that $[A,w] \in \left\langle \calA \right\rangle$.  Thus, 
 \[ [I, u] [A, w] = [A, uA + w] \neq [A, u + w] = [A,w][I,u], \]
 and hence $G_1^\ep$ is nonabelian.
 
 Let $x \in V$.  Then, we may write $x = w + u$, where $w \in \left\langle v \right\rangle$ and $u \in v^\perp$.  There is a unique $A \in H$ such that $[A,w] \in \calA$, and so $[A, x] = [A, w + u] = [A,w][I,u]$ is an element of $G_1^\ep$ such that $0^{[A,x]} = x$, and hence $G_1^\ep$ is transitive on $V$.  Since $|G_1^\ep| = |V|$, in fact $G_1^\ep$ acts regularly on $V$.
 
 Finally, since $A \in \SO^\ep(2m,q)$ for all $[A,x] \in G_1^\ep$, $G_1^\ep \le \ASO^\ep(2m,q)$, and, by Lemma \ref{lem:autaff}, $G_1^\ep$ is a subgroup of automorphisms of $\VO^\ep(2m, q)$, $\VNO^\ep(2m, q)$, and $\VNO_2^\ep(2m, q)$.  The result follows.
 \end{proof}

\begin{ex}
\label{ex:aff1}
 We can construct a concrete example for each $\ep$, $q$, and $m$.  When $\ep = 1$, for $\alpha \in \F_q$ we define 
 \[ C_\alpha \colonequals \begin{pmatrix}
            1      & 0       & 0 & \alpha   \\
            0      & 1       & 0 & -\alpha  \\
            \alpha & -\alpha & 1 & \alpha^2 \\
            0      & 0       & 0 & 1        \\   
                \end{pmatrix}, 
\]
and when $\ep = -1$, for $\alpha \in \F_q$ we define
\[ C_\alpha \colonequals \begin{pmatrix}
            1      & -\alpha^2 & \alpha & 0 \\
            0      & 1         & 0      & 0  \\
            0      & -2\alpha  & 1      & 0 \\
            0      & 0         & 0      & 1        \\   
                \end{pmatrix}. 
\]
Then, we may choose 
\[ A_\alpha \colonequals \left(\begin{array}{c|c}
                          I_{2m - 4} & 0        \\
                          \hline
                          0          & C_\alpha \\
                         \end{array}\right).
\]
Let $\{e_i : 1 \le i \le 2m\}$ be the standard basis for $V$.  Then, we may choose $v = e_1 + e_2$ if $\ep = 1$ and $v = e_2$ if $\ep = -1$.  

As another example, if $m > 2$, for $\alpha \in \F_q$, if 
\[ B_\alpha \colonequals \begin{pmatrix}
            1      & 0       & 0       & 0   \\
            0      & 1       & -\alpha & 0   \\
            0      & 0       & 1       & 0 \\
            \alpha & 0       & 0       & 1        \\   
                \end{pmatrix}, \]
then we may choose  
\[ A_\alpha \colonequals \left(\begin{array}{c|c}
                          B_\alpha   & 0        \\
                          \hline
                          0          & I_{2m-4} \\
                         \end{array}\right)
\]
with $v = e_5 + e_6$ (regardless of the value of $\ep$).  Thus, when $m > 2$, we may actually assume $G_1^{+} = G_1^{-}$.
\end{ex}
 
\begin{rem}
\label{rem:G1expp}
 Every element of $G_1^\ep$ can be expressed uniquely as $[A_\alpha, \alpha v + u]$, where $\alpha \in \F_q$ and $u \in v^\perp$.  Since
 \[ [A_\alpha, \alpha v + u]^p = \left[A_\alpha^p, p \alpha v + u \sum_{i = 0}^{p-1} A_{\alpha}^i \right] = \left[I, u(A_\alpha - I)^{p-1}\right],\]
 the choices of $A_{\alpha}$ from Example \ref{ex:aff1} show that, when $p > 3$ or $m > 2$, we can choose $G_1^\ep$ to have exponent $p$.  That such groups can be chosen to have exponent $3$ when $p = 3$ and $m = 2$ follows from direct inspection with \GAP \cite{GAP4}.
\end{rem} 
 
\subsection{A second family of PDSs in nonabelian groups of order $q^{2m}$} 
\label{subsec:secondfam}

The second family of PDSs requires a bit more care.  We will assume for this construction that either $m > 2$ or, if $m = 2$, $\ep = 1$.  As in the last example in Example \ref{ex:aff1}, for $\alpha \in \F_q$, if 
\[ B_\alpha \colonequals \begin{pmatrix}
            1      & 0       & 0       & 0   \\
            0      & 1       & -\alpha & 0   \\
            0      & 0       & 1       & 0 \\
            \alpha & 0       & 0       & 1        \\   
                \end{pmatrix}, \]
we then define  
\[ A_\alpha \colonequals \left(\begin{array}{c|c}
                          B_\alpha   & 0        \\
                          \hline
                          0          & I_{2m-4} \\
                         \end{array}\right).
\] 
(Again, we allow $m = 2$ as long as the quadratic form $Q$ is hyperbolic.)  In this case, $H \colonequals \{A_\alpha : \alpha \in \F_q\}$ is an elementary abelian group of order $q$ preserving the form $Q$.  Define 
\[ U \colonequals \left\langle e_1, e_4, \dots, e_{2m} \right\rangle.\] 
Direct calculation shows that $U$ is an $H$-invariant subspace of $V$.  

Define \[\calB \colonequals \left\langle [A_\alpha, \alpha e_2 + \beta e_3] : \alpha, \beta \in \F_q \right\rangle.\]

\begin{lem}
 \label{lem:Belab}
 The group $\calB$ is an elementary abelian group of order $q^2$.  In particular, for each $w \in \left\langle e_2, e_3\right\rangle$, there unique element $[A,x] \in \calB$ with $x = w$.
\end{lem}

\begin{proof}
 Noting that $H$ fixes $e_3$, a direct calculcation shows that, for all $\alpha, \beta, \gamma, \delta \in \F_q$, we have
 \[ [A_\alpha, \alpha e_2 + \beta e_3][A_\gamma, \gamma e_2 + \delta e_3] = [A_\gamma, \gamma e_2 + \delta e_3][A_\alpha, \alpha e_2 + \beta e_3] = [A_{\alpha + \gamma}, (\alpha + \gamma)e_2 + (\beta + \delta - \alpha \gamma)e_3].\]
 The result follows.
\end{proof}

Recalling that $U$ is $H$-invariant, for any $u \in U$, we have
\begin{align*} [A_\alpha, \alpha e_2 + \beta e_3]^{-1} [I, u] [A_\alpha, \alpha e_2 + \beta e_3] &= [A_\alpha^{-1}, -\alpha e_2 - (\alpha^2 + \beta)e_3] [I,u] [A_\alpha, \alpha e_2 + \beta e_3]\\ 
                        &= [I, uA_\alpha] \in T_U,\end{align*}
so $\calB$ normalizes $T_{U}.$

Define
\[ G_2\colonequals \left\langle T_U, \calB \right\rangle = T_{U}: \calB.\]

\begin{thm}
 \label{thm:nonabpds2}
 Let $m > 2$ or, if $m = 2$, then $\ep = 1$.  The group $G_2$ is a nonabelian group of order $q^{2m}$ in which the nonidentity elements can be partitioned into $D_0 \cup D_1 \cup D_2$, where each $D_i$ is a PDS, $\Cay(G_2, D_0) \cong \VO^\ep(2m, q)$, and $\Cay(G_2, D_1) \cong \Cay(G_2,D_2) \cong \VNO^\ep(2m, q)$.  
\end{thm}

\begin{proof}
 The proof is largely the same as that of Theorem \ref{thm:nonabpds}.  First, we have $|G_2| = q^{2m}$ since $G_2 = T_{U}:\calB$, $|T_{U}| = |U| = q^{2m-2}$, and $|\calB| = q^2$.       
 Moreover, not all vectors in $U$ are fixed by $H$; for example, $e_4A_1 = e_1 + e_4$, and so 
 \[ [I, e_4] [A_1, e_2] = [A_1, e_1 + e_2 + e_4 ] \neq [A_1, e_2 + e_4] = [A_1, e_2][I, e_4], \]
 and hence $G_2$ is nonabelian.
 
 Let $x \in V$.  Then, we may write $x = w + u$, where $w \in \left\langle e_2, e_3 \right\rangle$ and $u \in U$.  There is a unique $A \in H$ such that $[A,w] \in  \calB$, and so $[A, x] = [A, w + u] = [A,w][I,u]$ is an element of $G_2$ such that $0^{[A,x]} = x$, and hence $G_2$ is transitive on $V$.  Since $|G_2| = |V|$, in fact $G_2$ acts regularly on $V$.
 
 Finally, since $A \in \SO^\ep(2m,q)$ for all $[A,x] \in G_2$, $G_2 \le \ASO^\ep(2m,q)$, and, by Lemma \ref{lem:autaff}, $G_2$ is a subgroup of automorphisms of $\VO^\ep(2m, q)$, $\VNO^\ep(2m, q)$, and $\VNO_2^\ep(2m, q)$.  The result follows.
 \end{proof}

 \begin{rem}
  \label{rem:G2expp}
  A similar calculation to that done in Remark \ref{rem:G1expp} shows that $G_2$ has exponent $p$.
 \end{rem}

 \begin{thm}
  \label{thm:G1notG2}
  Let $m > 2$ or, if $m = 2$, then $\ep = +1$.  The groups $G_1^\ep$ and $G_2$ are not isomorphic.
 \end{thm}

 \begin{proof}
  If $m > 2$, we can choose $H$ to be the same group in each case.  If $W$ is the subspace of points fixed by $H$, then 
  \[ W = \left\langle e_1, e_3, e_5, \dots, e_{2m} \right\rangle,\]
  which has dimension $2m - 2$.  In each case, the central elements are of the form $[I, w]$, where $w \in W$.  Since $v = e_5 + e_6 \in W$, $|Z(G_1^\ep)| = q^{2m - 3}$.  On the other hand, $|Z(G_2)| = |W| = q^{2m - 2}$, which proves the claim for $m > 2$
  
  The proof is similar when $m = 2$ and $\ep = 1$: when $A_\alpha = C_\alpha$ (as in Example \ref{ex:aff1}), we see that the subspace of of points fixed by $H$ is $W = \left\langle e_1 + e_2, e_4\right\rangle$, and, since $v = e_1 + e_2$, 
  \[ Z(G_1^{+}) = \{[I, \beta e_4] : \beta \in \F_q\}.\]
  Thus, $|Z(G_1^{+})| = q$.  On the other hand, 
  \[ Z(G_2) = \{[I,w] : w \in W\},\]
  and so $|Z(G_2)| = q^2$, which proves the claim when $m = 2$ and $\ep = 1$. 
 \end{proof}

\subsection{A $(q+3)$-class amorphic association scheme in a group of order $q^4$} 
 
Let $V = GF(q)^4$, where $q$ is an odd prime.  Let $Q(x) = x_1x_2 + x_3x_4$, a hyperbolic form on $V$, and consider the group
\[ G \colonequals G_2 = T_U:\calB\]
defined in Subsection \ref{subsec:secondfam}, and define
\[ H\colonequals \{ B_\alpha : \alpha \in \F_q\}.\]

As in Theorem \ref{thm:nonabpds2}, we take $D_0$ to be the elements $[C, x]$ in $G$ where $Q(x) = 0$; $D_1$ to be the elements $[C,x]$ in $G$ where $Q(x)$ is a nonzero square; and $D_2$ to be the elements $[C,x]$ in $G$ where $Q(x)$ is a nonzero nonsquare.  Since each vector $x$ in $V$ occurs exactly once as the second component of an element $[C,x] \in G$, we may identify the elements of $G$ with the corresponding vector in the second component.  In other words, we identify $D_0$ with the set $V_0$ of vectors $x$ in $V$ such that $Q(x) = 0$, $D_1$ with the set $V_1$ of vectors $x$ in $V$ such that $Q(x)$ is a nonzero square, and $D_2$ with the set $V_2$ of vectors $x$ in $V$ such that $Q(x)$ is a nonsquare.  

\begin{lem}
\label{lem:partition}
The set $V_0$ can be partitioned into disjoint subsets of size $q^2 - 1$, where each subset is the set of nonzero vectors in a $2$-dimensional subspace of $V$.  Moreover, we can take each subset of the partition of $V_0$ to be $H$-invariant.  
\end{lem}

\begin{proof} 
 To see that we have a $H$-invariant partition for $V_0$, we define $v_\infty \colonequals e_4 = (0,0,0,1)$ and, for each $\alpha \in \F_q$, we define $v_\alpha \colonequals e_2 + \alpha e_4 = (0,1,0,\alpha)$.  Since $Q(v_\alpha) = 0$, both $v_\alpha, v_\alpha B \in V_0$ for all $B \in H$.  Moreover, if we define $u_\infty \colonequals e_1 = (1,0,0,0)$, $u_\alpha \colonequals \alpha e_1 - e_3 = (\alpha,0,-1,0)$ for $\alpha \in \F_q$, and $U_\alpha \colonequals \left\langle v_\alpha, u_\alpha\right\rangle$, for each $\alpha \in \F_q \cup \{\infty\}$ and $\beta \in \F_q$, we have
 \[ v_\alpha B_\beta = v_\alpha + \beta u_\alpha \in U_\alpha.\]  Since $u_\alpha \in \left\langle e_1, e_3 \right\rangle$, $u_\alpha B_\beta = u_\alpha$ for each $\alpha, \beta$, and thus each $U_\alpha$ is an $H$-invariant subspace.  It is routine to check that the $q+1$ subspaces $U_\alpha$, $\alpha \in \F_q \cup \{\infty\}$ are pairwise disjoint, and since $|V_0| = (q+1)(q^2 - 1)$, these subspaces form an $H$-invariant partition of $V$.
\end{proof}

\begin{prop}
\label{prop:VtoG}
 Let $D \subset V$ be a $H$-invariant subset of $V$, and view $(V,+)$ as the elementary abelian group of order $q^4$.  Then, $G$ is isomorphic to a regular subgroup of $\Aut(\Cay(V,D))$; that is, if 
 \[ D' \colonequals \{[C, x] \in G : x \in D\},\]
 then $\Cay(G, D') \cong \Cay(V,D)$.
\end{prop}

\begin{proof}
 Viewing $V$ additively, we then can define the graph $\Cay(V, D)$, where vectors $v$ and $w$ are adjacent iff $v - w \in D$.  For any $[C, x] \in G$, we have
\[ v^{[C,x]} - w^{[C,x]} = (vC + x) - (wC + x) = (v - w)C,\]
and so $v - w \in D$ iff $v^{[C,x]} - w^{[C,x]} \in D$.  Since $G$ is transitive on $V$ and preserves adjacency in $\Cay(V,D)$, the result follows.
\end{proof}

Recall the definitions of $D_1$ and $D_2$ in $G$ from above.  Let $\F_q \cup \{\infty\} = \{\alpha_3, \dots, \alpha_{q+3}\}$, define
\[ U_i \colonequals U_{\alpha_i}\]
as in the proof of Lemma \ref{lem:partition}, and define
\[ D_i \colonequals \left\{[C,x] \in G : x \in U_i - \{0\} \right\}.\]

\begin{thm}
 \label{thm:affassocscheme}
 Each $D_i$, $1 \le i \le q + 3$, is a PDS of Latin square type in $G$.  Consequently, $\{D_i : 1 \le i \le q + 3\}$ corresponds to a $(q+3)$-class amorphic association scheme, and a union of any number of these PDSs is also a PDS of Latin square type in $G$. 
\end{thm}

\begin{proof}
 First, $D_1$ and $D_2$ are PDSs of Latin square type in $G$ with $r = q(q-1)/2$ by Theorem \ref{thm:nonabpds2}.
 
 Since each $U_i$ is a subspace of size $q^2$, each graph $\Cay(V, U_i - \{0\})$ is a union of disjoint complete subgraphs, i.e., each graph $\Cay(V, U_i - \{0\})$ is a $(q^4, q^2 - 1, q^2 - 2, 0)$-strongly regular graph of Latin square type.  By Proposition \ref{prop:VtoG}, this means each $D_i$, $3 \le i \le q + 3$ is also a PDS of Latin square type.  Finally, since $\{D_i : 3 \le i \le q + 3\}$ is a partition of $D_0$ and $\{D_0, D_1, D_2\}$ is a partition of $G$, $\{D_i : 1 \le i \le q + 3\}$ is a $(q+3)$-class amorphic association scheme.  The result follows from Corollary \ref{AmorphicPDS}.
\end{proof}

\begin{cor}
 \label{cor:affpaley}
 The group $G$ contains a Paley-type PDS.
\end{cor}

\begin{proof}
 Define
 \[D \colonequals D_1 \cup \bigcup_{i = 3}^{(q+5)/2} D_i, \]
 i.e., $D$ is the union of $D_1$ and half of the $D_i$'s, where $i \ge 3$.
 By Theorem \ref{thm:affassocscheme}, $D$ is a PDS of Latin square type, and, since $D$ contains the elements of $D_1$ and exactly half of the elements of $D_0$,
 \[ |D| = \frac{q(q-1)(q^2-1)}{2} + \frac{(q+1)(q^2 - 1)}{2} = \frac{q^4 - 1}{2}.\]
 The result follows.
\end{proof}

 
 \section{Partial difference sets in semidirect products with a large center}
 \label{sec:SemidirectLargeCenter}
 
 Let $p$ be a prime, and define $G \colonequals \left\langle x, y : x^{p^2} = y^{p^2} = 1, xy = yx \right\rangle \cong \Z_{p^2}^2$. The following sets were shown in~\cite{Davis_1994} to be $(p^4, p(p^2-1), 2p^2-3p,p^2-p)$-PDSs in $G$ for $1 \le i  \le p-1$:
 
 \[P_i = \left( \bigcup_{j=0}^{p-1} \left(\left\langle xy^{j+pi} \right\rangle - \left\langle x^p y^{pj} \right\rangle\right) \right) \cup \left(\left\langle x^{pi}y \right\rangle - \left\langle y^p \right\rangle\right) .\]
 
The following subgroups with the identity removed are trivial $(p^2,p^2-1,p^2-2,0)$-PDSs:
 
 $$ S_j = \left\langle xy^j \right\rangle - \{ 1 \} \mbox{ for }  0 \le j \le p-1,  S_{\infty} =  \left\langle y \right\rangle - \{ 1 \}.$$
 
 The $P_i$ and $S_j$ partition the nonidentity elements of $G$ into Latin-square type PDSs.  The next theorem applies Theorem~\ref{VDtheorem} to this collection.
 
  \begin{thm}
 \label{semidirectamorphicscheme}
The collection $\{ P_1, P_2,...,P_{p-1},S_0,S_1,\dots,S_{p-1},S_{\infty} \} $ is a $2p$-class amorphic association scheme on $G$.
 \end{thm}

Combining Corollary~\ref{AmorphicPDS} with Theorem~\ref{semidirectamorphicscheme} implies that \[D \colonequals \left(\bigcup_{j=1}^{\frac{p-1}{2}} P_i \right) \cup \left(\bigcup_{j=0}^{\frac{p-1}{2}} S_j\right)\] is a Paley-type $\left(p^4, \frac{p^4-1}{2}, \frac{p^4-5}{4}, \frac{p^4-1}{4}\right)$-PDS.

 This construction was the first known PDS with Paley-type parameters in a group that was not elementary abelian; other abelian PDSs with these parameters have since appeared (see, for instance, \cite{Polhill_2002}). 
 
 We now show that a similar construction will produce Paley-type PDSs in certain nonabelian groups, which along with the construction in the previous section (Corollary \ref{cor:affpaley}) are the first such constructions of Paley-type PDSs in nonabelian groups known to these authors. Consider the group \[\widehat{G_2} \colonequals \left\langle x, y : x^{p^2} = y^{p^2} = 1, yxy^{-1} = x^{p^2-p+1} \right\rangle \cong \Z_{p^2} \rtimes_{p^2-p+1} \Z_{p^2}.\] Define \[\widehat{P_i} \colonequals \left(\bigcup_{j=0}^{p-1} \left(\left\langle xy^{j+pi} \right\rangle - \left\langle x^p y^{pj} \right\rangle\right)\right) \cup \left(\left\langle x^{pi}y \right\rangle - \left\langle y^p \right\rangle\right) \] and $ \widehat{S_j} \colonequals \left\langle xy^j \right\rangle - \{ 1 \}$  for $0 \le j \le p-1$,  $\widehat{S_{\infty} } \colonequals \left\langle y \right\rangle - \{ 1 \}$, and finally\[\widehat{D_2}\colonequals  \left(\bigcup_{j=1}^{\frac{p-1}{2}} \widehat{P_i} \right) \cup \left(\bigcup_{k=0}^{\frac{p-1}{2}} \widehat{S_k}\right).\] We note that the formal sets $P_i$ and $\widehat{P_i}$ appear the same, but the element $(xy^{p+1})^2 = x^2y^{2p+2} \in P_1$ whereas $(xy^{p+1})^2  = x^{p^2+2}y^2 \in \widehat{P_1}$. 
 
 In order to demonstrate that $\widehat{P_i}$ is a PDS, we first prove two lemmas that we will use in the proof of the result.
 
 \begin{lem}
 \label{internaldifferences}
 Let $1 \le i \le (p-1)$. For the group $\widehat{G_2}$ and the subset $\widehat{P_i}$ defined above, we have \[ \sum_{j=0}^{p-1} \left((p^2-2p) \left\langle xy^{ip+j} \right\rangle + p \left\langle x^p y^{pj} \right\rangle\right) + (p^2-2p) \left\langle x^{ip} y \right\rangle + p \left\langle y^p \right\rangle\] 
 \[= (p^2-p)(p+1) 1_{\widehat{G_2}} + (p^2-2p) \widehat{P_i} + (p^2-p) \left\langle x^p,y^p \right\rangle.\]
 \end{lem}
 
 \begin{proof}
 All of the elements of $\widehat{P_i}$ of order $p^2$ will appear $(p^2-2p)$ times, and all of the elements of $\left\langle x^p, y^p \right\rangle$ will appear an additional $p$ times.
 \end{proof}
 
 \begin{lem}
 \label{externaldifferences}
 For $j, j' \in \{ 0,1,\ldots, p-1 \}$ we have 
 \[\sum_{j \neq j'} \left(\left\langle xy^{ip+j} \right\rangle - \left\langle x^p y^{pj} \right\rangle\right) \left(\left\langle xy^{ip+j'} \right\rangle - \left\langle x^p y^{pj'} \right\rangle\right) + \sum_{j=0}^{p-1} \left(\left\langle xy^{ip+j} \right\rangle - \left\langle x^p y^{pj} \right\rangle\right) \left(\left\langle x^{ip+j}y \right\rangle - \left\langle x^{pj} y^{p} \right\rangle\right)\] 
 \[= (p^2-p)\widehat{G_2} - (p^2-p) \left\langle x^p, y^p \right\rangle.\]
 \end{lem}
 
 \begin{proof}
 A symmetry argument implies that all of the elements of order $p^2$ appear the same number of times in this sum, and the fact that these are different values of $j, j'$ implies that we will not get any elements of order $p$. The result follows from a counting argument.
 \end{proof}
 
 \begin{thm}
 \label{nonabelianPaley}
Let $p$ be a prime. The collection $\{ \widehat{P_1}, \widehat{P_2},\dots,\widehat{P_{p-1}},\widehat{S_0},\widehat{S_1},..,\widehat{S_{p-1}},\widehat{S_{\infty}} \} $ is a $2p$-class amorphic association scheme on $\widehat{G_2}$ and the set $\widehat{D_2}$ is a Paley-type $\left(p^4, \frac{p^4-1}{2}, \frac{p^4-5}{4}, \frac{p^4-1}{4}\right)$-PDS in $G_2$.
 \end{thm}
 
 \begin{proof}
 Since the $\widehat{S_i}$ are all subgroups, they are all (trivial) Latin-square type PDSs, and Lemmas~\ref{internaldifferences} and~\ref{externaldifferences} imply the following.
 
 \begin{align*} 
 \widehat{P_i}^2  &= &&\sum_{j=0}^{p-1} \left(\left\langle xy^{ip+j} \right\rangle - \left\langle x^p y^{pj} \right\rangle\right)^2 + \left(\left\langle x^py \right\rangle - \left\langle y^p \right\rangle\right)^2\\
 & &&+ \sum_{j \neq j'} \left(\left\langle xy^{ip+j} \right\rangle - \left\langle x^p y^{pj} \right\rangle\right) \left(\left\langle xy^{ip+j'} \right\rangle - \left\langle x^p y^{pj'} \right\rangle\right)\\
 & &&+ \sum_{j=0}^{p-1} \left(\left\langle xy^{ip+j} \right\rangle - \left\langle x^p y^{pj} \right\rangle\right) \left(\left\langle x^{ip+j}y \right\rangle - \left\langle x^{pj} y^{p} \right\rangle\right) \\
 &= &&(p^2-p)(p+1) 1_{\widehat{G_2}} + (p^2-2p) \widehat{P_i} + (p^2-p) \left\langle x^p,y^p \right\rangle + (p^2-p)\widehat{G_2} - (p^2-p) \left\langle x^p, y^p \right\rangle \\
 &= &&(p^3-p) 1_{\widehat{G_2}} +  (p^2-2p) \widehat{P_i} + (p^2-p) \widehat{G_2} \\
 &= &&(p^3-p) 1_{\widehat{G_2}} +  (2p^2-3p) \widehat{P_i} + (p^2-p) (\widehat{G_2}- \widehat{P_i} - 1_{\widehat{G_2}}) \\
 \end{align*}
Thus, the $\widehat{P_i}$ are all $(p^4, p^3-p, 2p^2-3p, p^2-p)$-PDSs in $\widehat{G_2}$ as claimed. Since these are all Latin-square type PDSs, Corollary~\ref{AmorphicPDS} implies that any union of these PDSs will be a PDS. In particular, \[\widehat{D_2} = \left(\bigcup_{j=1}^{\frac{p-1}{2}} \widehat{P_i}\right) \cup \left(\bigcup_{k=0}^{\frac{p-1}{2}} \widehat{S_i}\right)\] is a $\left(p^4, \frac{p^4-1}{2}, \frac{p^4-5}{4}, \frac{p^4-1}{4}\right)$-PDS as required. \end{proof}

We now turn to a generalization of this construction. Let \[G_t = \left\langle x, y : x^{p^t} = y^{p^t} = 1, xy = yx \right\rangle \cong \Z_{p^{t}} \times \Z_{p^{t}}.\] Polhill~\cite{Polhill_2002} showed that the sets $P_{t,i}$ are  Latin-square type PDSs with parameters \[\left(p^{2t}, \frac{p^{t}-p}{p-1} (p^t-1), p^t + (\frac{p^t-p}{p-1})^2 - 3 \frac{p^t-p}{p-1},(\frac{p^t-p}{p-1})^2 - 3 \frac{p^t-p}{p-1}\right)\] in $G_t$ for $1 \le i \le p-1$, where $P_{t,i}$ is defined to be  

\[\bigcup_{r=1}^{t-1} \left(\bigcup_{j=0}^{p^{r-1}-1} \left(\bigcup_{k=0}^{p-1} \left(\left\langle xy^{ip^r + pj + k} \right\rangle - \left\langle x^{p^{t-r}} y^{jp^{t+1-r}  + kp^{t-r}} \right\rangle\right)\right) \cup \left(\left\langle x^{ip^r + jp} y \right\rangle - \left\langle x^{jp^{t-r+1}} y^{p^{t-r}} \right\rangle\right)\right).
\]

\noindent If we define $S_{t,j}\colonequals \left\langle xy^j \right\rangle - \{ 1_{G_t} \}, 0 \le j \le p-1, $ and $S_{t,\infty}: = \left\langle y \right\rangle - \{ 1_{G_t} \}$, then we get the following theorem, which is analogous to Theorem~\ref{semidirectamorphicscheme}.

\begin{thm}
\label{semidirectamorphicscheme_t}
For $t \ge 2$, the collection $\{ P_{t,1}, P_{t,2}, \ldots ,P_{t,p-1},S_{t,0},S_{t,1}, \ldots ,S_{t,p-1},S_{t,\infty} \}$ is a $2p$-class amorphic association scheme on $G_t$.
 \end{thm}

The combination of Corollary~\ref{AmorphicPDS} and Theorem~\ref{semidirectamorphicscheme_t} imply that \[D_t\colonequals\displaystyle \left(\bigcup_{i=1}^{\frac{p-1}{2}} P_{t,i}\right) \cup \left(\bigcup_{j=0}^{\frac{p-1}{2}} S_{t,j}\right)\] is a Paley-type $\left(p^{2t}, \frac{p^{2t}-1}{2}, \frac{p^{2t}-5}{4}, \frac{p^{2t}-1}{4}\right)$-PDS in $G_t$.

We now construct Paley-type PDSs in the nonabelian group 
\[\widehat{G_t} \colonequals \left\langle x, y : x^{p^t} = y^{p^t} = 1, yxy^{-1} = x^{(p-1)p^{t-1} + 1} \right\rangle \cong \Z_{p^t} \rtimes_{(p-1)p^{t-1} + 1} \Z_{p^t}.\] To do this, we define a collection of disjoint PDSs that partition the nonidentity elements of $\widehat{G_t}$ in an analogous fashion as those defined in $G_t$: first, we define $\widehat{P_{t,i}}$ to be 
\[\bigcup_{r=1}^{t-1} \left(\bigcup_{j=0}^{p^{r-1}-1} \left(\bigcup_{k=0}^{p-1} \left(\left\langle xy^{ip^r + pj + k} \right\rangle - \left\langle x^{p^{t-r}} y^{jp^{t+1-r}  + kp^{t-r}} \right\rangle\right)\right) \cup \left(\left\langle x^{ip^r + jp} y \right\rangle - \left\langle x^{jp^{t-r+1}} y^{p^{t-r}} \right\rangle\right)\right);
\]
then we define 
\begin{eqnarray*}
\widehat{S_{t,j}} & \colonequals & \left\langle xy^j \right\rangle - \{ 1_{\widehat{G_t}} \}, \\
\widehat{S_{t,\infty}} & \colonequals & \left\langle y \right\rangle - \{ 1_{\widehat{G_t}} \}. \\
\end{eqnarray*}

The main construction in this section is the following, and along with those examples in the previous section are the first examples of Paley-type PDSs in nonabelian groups known to the authors.
\begin{thm}
\label{mainPaleyconstr}
For $t \ge 2$, the collection $\{ \widehat{P_{t,1}}, \widehat{P_{t,2}}, \ldots ,\widehat{P_{t,p-1}},\widehat{S_{t,0}},\widehat{S_{t,1}}, \ldots ,\widehat{S_{t,p-1}},\widehat{S_{t,\infty}} \}$ is a $2p$-class amorphic association scheme on $\widehat{G_t}$. Therefore, 
\[\widehat{D_t}\colonequals \left(\bigcup_{i=1}^{\frac{p-1}{2}} \widehat{P_{t,i}}\right) \cup \left(\bigcup_{j=0}^{\frac{p-1}{2}} \widehat{S_{t,j}}\right)\] is a Paley-type $\left(p^{2t}, \frac{p^{2t}-1}{2}, \frac{p^{2t}-5}{4}, \frac{p^{2t}-1}{4}\right)$-PDS in $\widehat{G_t}$.
 \end{thm}
 
 The proof uses the same reasoning as the proof of Theorem~\ref{nonabelianPaley} and is left for the reader.


\section{Paley-type PDSs and Paley-Hadamard DSs in nonabelian groups}
\label{sec:Paley}

In this section, we will use results from the previous sections to construct additional examples of nonabelian Paley-type PDSs as well as nonabelian Stanton-Sprott (Twin prime power) Paley-Hadamard DSs. Davis~\cite{Davis_1994} used character theory to prove a product construction for abelian groups; we will show that the theorem remains true for nonabelian groups. The theorem will enable us to recursively build nonabelian PDSs with Paley-type parameters.

\begin{thm}
\label{generalproduct}
Suppose that the groups $G$ and $G'$ of order $v$ both possess PDSs of the Paley-type having parameters $\left(v,\frac{v-1}{2},\frac{v-5}{4},\frac{v-1}{4}\right)$, $D$ and $D'$ respectively.  Then,the group ${\cal G} \colonequals G \times G'$ also contains a Paley-type PDS with parameters $\left(v^2,\frac{v^2-1}{2},\frac{v^2-5}{4}.\frac{v^2-1}{4}\right)$.
\end{thm}
\begin{proof}
If $D$ and $D'$ are Paley-type PDSs in $G$ and $G'$, respectively, then $D^c=G-1_G-D$ and $D'^c = G' - 1_{G'} - D'$ are also Paley-type PDSs in $G$ and $G'$, respectively.  The following group ring equations then hold in $G$ and $G'$ as a consequence of the sets $D, D^c, D', D'^c$ being PDSs:

$$D D^c = D^c D =\frac{v-1}{4} G^*,$$
$$D'D'^c = D'^cD' =\frac{v-1}{4} G'^*,$$
where $G^*$ and $G'^*$ denote $G - 1_{G}$ and $G' - 1_{G'}$, respectively.

Our Paley-type PDS in ${\cal G} = G \times G'$ is given by ${\cal D}=D(1+D')+D^c(1+D'^c)$, as verified in the following group ring computation:

\begin{eqnarray*}
{\cal D}^2 & = & (D(1+D') + D^c(1+D'^c))^2 \\
& = & D^2 (1+D')^2 + 2DD^c(1+D')(1+D'^c) + (D')^2(1+D'^c)^2 \\
& = & \left(\frac{v-5}{4}D + \frac{v-1}{4}D^c + \frac{v-1}{2}1_G\right)\left( \frac{v+3}{4}D' + \frac{v-1}{4}D'^c + \frac{v+1}{2}1_{G'}\right) \\
& & + \frac{v-1}{2}G^*\left(1+\frac{v+3}{4}{G'}^*\right) \\
& & + \left(\frac{v-1}{4}D + \frac{v-5}{4}D^c + \frac{v-1}{2}1_G\right)\left( \frac{v-1}{4}D' + \frac{v+3}{4}D'^c + \frac{v+1}{2}1_{G'}\right) \\
& = & \frac{v^2-1}{4}1_{{\cal G}} + \frac{v^2-5}{4}{\cal D} + \frac{v^2-1}{4}({\cal G}-{\cal D}). \\
\end{eqnarray*}
\end{proof}

To illustrate the scope of Theorem~\ref{generalproduct}, consider the groups $G = \Z_{25} \rtimes_{21} \Z_{25}$ and $G' = G_2 = \left\langle T_U, \calB \right\rangle$ (with $q = 5$ and $m = 2$) from the discussion before Theorem \ref{thm:nonabpds2}. Both of these groups have $(625, 312, 155, 156)$-PDSs and hence ${\cal G} = G \times G'$ will have a $\left(5^8, \frac{5^8-1}{2}, \frac{5^8-5}{4}, \frac{5^8-1}{4}\right)$-PDS. We can continue to apply the theorem by first constructing a $\left(5^8, \frac{5^8-1}{2}, \frac{5^8-5}{4}, \frac{5^8-1}{4}\right)$-PDS in ${\cal G}' = \Z_{25}^2 \times \Z_5^4$ (both $\Z_{25}^2$ and $\Z_5^4$ have $(625, 312, 155, 156)$-PDSs, and Theorem~\ref{generalproduct} implies that their product will have a PDS), and we can then apply Theorem~\ref{generalproduct} to get a $\left(5^{16}, \frac{5^{16}-1}{2}, \frac{5^{16}-5}{4}, \frac{5^{16}-1}{4}\right)$-PDS in ${\cal G} \times {\cal G}'$. Repeated uses of the Theorem give constructions of Paley-type PDSs in groups of the form $G^{2^t}, {G'}^{2^t},{({\cal G} \times {\cal G}')}^{2^t}$, and ${\cal G}^{2^t}$. As long as the sizes of the groups are the same, we can repeatedly apply Theorem~\ref{generalproduct} to get Paley-type PDSs in larger groups. One general example of a family with a variety of exponents for the constituent groups is the following.

\begin{cor}
\label{Paleyspgroups}
The group $\Z_p^4 \times (\Z_{p^2} \rtimes_{p^2-p+1} \Z_{p^2}) \times (\Z_{p^4} \rtimes_{p^4-p^3+1} \Z_{p^4}) \times \cdots \times (\Z_{p^{2^t}} \rtimes_{p^{2^t}-p^{2^t-1}+1} \Z_{p^{2^t}})$ has a Paley-type PDS for all $t \ge 2$. 
\end{cor}

Paley-type PDSs can in turn be used to generate Paley-Hadamard DSs using the Stanton-Sprott construction \cite{Stanton_Sprott_1958}.  As with the recursive construction for Paley-type PDSs, we show that the input groups need not be abelian. Since we now have constructions of nonabelian Paley-type PDSs, we will be able to construct new Paley-Hadamard DSs that are nonabelian. To our knowledge, these are the first nonabelian DSs with these parameters.

\begin{thm}
\label{twinprimes}
Suppose that the group $G$ contains a Paley-type $\left(v, \frac{v-1}{2}, \frac{v-5}{4}, \frac{v-1}{4}\right)$-PDS and the group $G'$ contains a skew Hadamard $\left(v \pm 2,\frac{(v \pm 2)-1}{2},\frac{(v \pm 2)-3}{4}\right)$-DS.  Then,the product group $G \times G'$ contains a Paley-Hadamard DS in the Stanton-Sprott (Twin prime power) family.
\end{thm}
\begin{proof}
We will prove the case where $|G'| = v+2$, with the $v-2$ case being extremely similar.
If $D$ is a Paley-type $\left(v,\frac{v-1}{2},\frac{v-5}{4},\frac{v-1}{4}\right)$-PDS in $G$, then $D^c=G-1_G-D$ is also a Paley-type PDS. We use the facts from the proof of Theorem~\ref{generalproduct} together with the similar equations for the skew-Hadamard DS $D'$ in $G'$ to get the following.


The set ${\cal D} \colonequals G + DD' + D^cD'^{(-1)} \subset G \times G'$ is a DS as verified below:

\begin{eqnarray*}
{\cal D} {\cal D}^{(-1)} &= & \left(G + DD' + D^cD'^{(-1)}\right)\left(G + DD' + D^cD'^{(-1)}\right)^{(-1)} \\
&= & \left(G + DD' + D^cD'^{(-1)}\right)\left(G + DD'^{(-1)} + D^cD'\right) \\
&= & G^2 + (GD)D'^{(-1)}+ (GD^c)D'+(GD)D'+{D}^2(D'D'^{(-1)})+(DD^c)D'^2 \\
& & +(GD')D'^{(-1)}+(DD^c)(D'^{(-1)})^2+{D^c}^2(D'^{(-1)}D) \\
&= & vG(1_{G'})+(v-1)G(D' + D'^{(-1)}) \\
& & + \left(\frac{v-1}{4}G^*\right)\left(\frac{v-3}{4}+\frac{v-1}{4}\right){G'}^* \\
& & +\left(\frac{v-1}{4}G^*\right)\left(\left(\left(\frac{v-5}{4}+\frac{v-1}{4}\right)G^*+(v-1)1_G\right)\left(\frac{v+1}{2}1_{G'}+\frac{v-1}{4}{G'}^*\right)\right).
\end{eqnarray*}




 

Combining terms leads to the equation

$${\cal D}{\cal D}^{(-1)} = \frac{v^2+2v-3}{4}{\cal G}^* + \frac{v^2+2v-1}{2}1_{{\cal G}},$$

thus proving the result.
\end{proof}


Recall the group $G_2$ from Section \ref{sect:aff}, letting $q = 3$ and $m = 2$. As examples of Theorem~\ref{twinprimes}, the nonabelian groups  $G_2 \times \Z_{83}$ and $(\Z_9 \rtimes_{7} \Z_9) \times \Z_{83}$ each have a $(6723, 3361, 1680)$-difference set; the nonabelian groups $(G_2)^2 \times \Z_{6563}$, $(\Z_9 \rtimes_7 \Z_9) \times G_2 \times \Z_{6563}$, $(\Z_9 \rtimes_7 \Z_9)^2 \times \Z_{6563}$, and $(\Z_{81} \rtimes_{55} \Z_{81}) \times \Z_{6563}$ have $(45724643, 22862321, 11431160)$-difference sets; and the nonabelian group $(\Z_{27} \rtimes_{19} \Z_{27}) \times \Z_{727}$ has a $(529983, 264991, 132495)$-difference set. A more general corollary is the following (although there will be many nonabelian groups containing a Paley-Hadamard difference set that are not contained in this result).

\begin{cor}
Let $r \ge 2$. If $q = p^{2^r} \pm 2$ is prime, then the nonabelian group \[(\Z_{p^{2^{r}}} \rtimes_{(p-1)p^{2^{r-1}-1}+1} \Z_{p^{2^{r}}}) \times \Z_q\] has a $\left(qp^{2^{r}}, (qp^{2^{r}}-1)/2, (qp^{2^r}-3)/4\right)$-difference set.
\end{cor}

\section{Many product theorems allow nonabelian groups}
\label{sec:product}
In the previous section, we used group rings to prove that two results previously known in abelian groups also hold in nonabelian groups. In this section, we will show that in some cases we can avoid the quadratic group ring calculations entirely because the relations needed to simplify the calculations do not depend on whether the group is abelian or not. 

\begin{lem} 
\label{GroupRingRelationsLemma}
Suppose the group $G$ has a partition of the nonidentity elements into PDSs $P_1, P_2,...,P_n$ all of the Latin square type or all of the negative Latin square type.  Then,the quadratic group ring relations relating the $P_i$ and $P_j$ are strictly determined by the parameters and do not depend on whether the group is abelian or not.

\end{lem}
\begin{proof}
 Let $P_i$ be a $(v,k_i,{\lambda}_i,{\mu}_i)$-PDS. Then:
$${P_i}^2 = (k_i-\mu_i)1_G + \lambda_i P_i + \mu_i (G - P_i).$$
Now suppose that $P_i$ and $P_j$ are part of the partition of $G$ into Latin or negative Latin square type PDSs.  Let $P_i$ be a $(v,k_i,{\lambda}_i,{\mu}_i)$-PDS and $P_i$ be a $(v,k_j,{\lambda}_j,{\mu}_j)$-PDS.
The key to getting a relation for $P_iP_j$ is the fact that the union of disjoint PDSs of Latin (alternatively negative Latin square type) will also be a PDS of the same type by Corollary \ref{AmorphicPDS}.  The same corollary ensures $P_iP_j = P_jP_i$.

Therefore, we can write two equations for $(P_i+P_j)^2$, the first of which is expanding and using the individual PDS parameters:
\begin{align*}
(P_i+P_j)^2 &= P_iP_j+ P_jP_i+ {P_i}^2 + {P_j}^2\\
            &=2P_iP_j+ (k_i-\mu_i)1_G + \lambda_i P_i + \mu_i (G - P_i)+(k_j-\mu_j) + \lambda_j P_j + \mu_j (G - P_j).
\end{align*}

Now we will use the fact that $P_i \cup P_j$ is a $(v,k_i+k_j,\lambda,\mu)$-PDS.
$${(P_i+P_j)}^2 = (k_i+k_j-\mu)1_G + \lambda (P_i+P_j) + \mu (G - P_i-P_j).$$

Setting the equations equal and solving yields:
$$P_iP_j = (\lambda -\lambda_i -\mu_j)P_i + (\lambda-\mu_i-\lambda_j)P_j + (\mu - \mu_i -\mu_j)(G-1-P_i-P_j).$$
Hence, the relations for both ${P_i}^2$ and $P_iP_j$ are determined by the parameters.  Suppose that $G$ has a partition of the nonidentity elements into PDSs $P_1, P_2,...,P_n$ and $G'$ has a partition of the nonidentity elements into PDSs $P'_1, P'_2,...,P'_n$ where $P_i$ and $P'_i$ have the same parameters. Then,the relations for 
${P_i}^2$ and ${P'_i}^2$ are the same for $P_i$ relative to $G$ as for $P'_i$ relative to $G'$ and furthermore the relations for 
$P_iP_j$ and $P'_iP'_j$  are the same for $P_i$ and $P_j$ relative to $G$ as for $P'_i$ and $P'_j$ relative to $G'$. The result follows.
\end{proof}

\begin{thm}
\label{Productstononabelian}
Suppose that there is a group $G$ having a partition $
\calP = \{ 1_G, P_1,P_2,...,P_n \}$ where all the $P_i$ are Latin square type PDSs or negative Latin square type PDSs and $1_G$ is the identity in $G$. Let $G'$ be any other group, and suppose $D$ is a PDS in $G \times G'$ which is constructed as a union of sets of the form $A_ix'_i$, $A_i \in \calP, x'_i  \in G'$.  Suppose there is another group $\widehat{G}$ that has a partition of its nonidentity elements into PDSs $\{ \widehat{P_1}, \widehat{P_2},...,\widehat{P_n} \}$ where $P_i$ and $\widehat{P_i}$ have the same parameters for all $i$. If we define the set 
$\widehat{D}$ by replacing each $P_i$ with $\widehat{P_i}$ and $1_G$ with $1_{\widehat{G}}$, then $\widehat{D}$ will be a PDS in  $\widehat{G} \times G'$ with the same parameters as $D$. 
\end{thm}	
\begin{proof}
Let $D$ be a $(v,k,\lambda, \mu)$-PDS in $G \times G'$ so that:
$$DD^{(-1)}= (k-\mu)1_{G \times G'} + \lambda D + \mu (G \times G'-D).$$

Consider $\widehat{D}{\widehat{D}}^{(-1)}$. Each term in the expansion will have one of the following forms:
\begin{itemize}
\item[1.] $(1_{\widehat{G} }x' )(1_{\widehat{G}} y') = 1_{\widehat{G} }x' y'$ where $x',y' \in G'$. In our calculation of $DD^{(-1)}$ we have the corresponding term $(1_G x' )(1_G y') = 1_Gx' y'$.
\item[2.] $(1_{\widehat{G} }x' )(\widehat{P_i} y') = \widehat{P_i} x' y'$ where $x',y' \in G'$. In our calculation of $DD^{(-1)}$ we have the corresponding term $(1_G x' )(P_i y') = P_ix' y'$.
\item[3.] $(\widehat{P_i} x')(\widehat{P_i} y') ={\widehat{P_i}}^2x' y'=((k_i-\mu_i)1_{\widehat{G}} + \lambda_i \widehat{P_i} + \mu_i (\widehat{G} - \widehat{P_i}))x'y'$ where $x',y' \in G'$, In our calculation of $DD^{(-1)}$ we have the corresponding term $(P_i x' )(P_i y') = {P_i}^2x' y'=(k_i-\mu_i)1_G + \lambda_i P_i + \mu_i (G - P_i)x'y'$, where both $P_i$ and $\widehat{P_i}$ are $(v,k_i,\lambda_i,\mu_i)$-PDSs.
\item[4.] $(\widehat{P_i}x')(\widehat{P_j}y') = \widehat{P_i}\widehat{P_j} (x',y')$ where $x',y' \in G'$. In our calculation of $DD^{(-1)}$ we have the corresponding term $(P_ix')(P_jy') = (P_iP_j)(x'y')$.   By the preceding lemma, we know that $P_iP_m = aP_i + bP_j + c(G-1_G-P_i-P_j)$ means that $\widehat{P_i}\widehat{P_m} = a\widehat{P_i} + b\widehat{P_j} + c(G-1_{\widehat{G}}-\widehat{P_i}-\widehat{P_j})$ and vice versa, since the quadratic group ring equations relating the $P_i$ are determined by the parameters.  
\end{itemize}
Therefore, when we expand $\widehat{D}{\widehat{D}}^{(-1)}$ we will the exact same count of terms for $\widehat{P_i} x'$ and $1_{\widehat{G}}x'$ for any $x' \in G'$ as we would respectively for $P_ix'$ and $1_Gx'$ when calculating $DD^{(-1)}$. It follows that:
$$\widehat{D}{\widehat{D}}^{(-1)} = (k-\mu)1_{\widehat{G} \times G'} + \lambda \widehat{D} + \mu (\widehat{G} \times G' -\widehat{D}).$$
\end{proof}

As it so happens, there are rather many product constructions that fit the hypotheses for the theorem including all the product theorems in \cite{Davis_Xiang_2004, Polhill_2008, Polhill_2019, Polhill_2009a, Polhill_etal_2013}.  We will use a particular construction in 3-groups that incorporates the examples from Sections \ref{sect:aff} and \ref{sec:SemidirectLargeCenter} to give the reader an idea of what can be done with these generalizations.

\begin{subsection}{3-groups}
In the group $G=\Z_3 \times \Z_3 = \langle x,y \rangle$, we can partition the nonidentity elements into four trivially intersecting subgroups of order $\sqrt{|G|}=3$ with the identity removed: $H_1 = \{ x,x^2 \}, H_2 = \{ xy, x^2y^2 \}, H_3 = \{ xy^2, x^2y \}, $ and $H_4 = \{ y, y^2 \} $.  We can then produce a partition of Latin square type PDSs as \[L_0 \colonequals H_3 \cup H_4, \; L_1 \colonequals H_1,\; L_2 \colonequals H_2\] or a negative Latin square type partition as \[C_0 \colonequals \varnothing, \;C_1 \colonequals H_1 \cup H_2,\; C_2 \colonequals H_3 \cup H_4.\]  

Now we consider groups order 81. From Section \ref{sect:aff} we can obtain a partition of the nonidentity elements in the groups ${G_1}^+$ and $G_2$ into three PDSs $L_0, L_1,$ and $L_2$ of cardinalities $32, 24,$ and $24$, respectively, and similarly the group ${G_1}^{-}$ into three PDSs $C_0, C_1,$ and $C_2$ of cardinalities $20, 30,$ and $30$.  From Section \ref{sec:SemidirectLargeCenter} we can obtain a partition of the nonidentity elements of either $\Z_{9} \times \Z_{9}$ or $\Z_{9} \rtimes_{7} \Z_{9}$ into three PDSs $L_0, L_1,$ and $L_2$ of cardinalities $32, 24,$ and $24$.  

Using these partitions and applying Theorem \ref{Productstononabelian} to \cite[Theorems 2.1--2.3]{Polhill_2008} in the case of $p=3$ updates \cite[Corollary 5.1]{Polhill_2008} and gives us infinite families of Latin and negative Latin square type PDSs that now include nonabelian groups.

\begin{cor}
Suppose $G$ with $|G| = 3^{2m}$ has a partition of its nonidentity elements into three Latin square type PDSs $L_0,L_1,$ and $L_2$ of cardinality $(3^{m-1}+1)(3^m-1), (3^{m-1})(3^m-1),$ and $(3^{m-1})(3^m-1)$ respectively.  Suppose also that $G'$ with $|G'| = 3^{2n}$ has a partition of its nonidentity elements into three Latin square type PDSs $C_0,C_1,$ and $C_2$ of cardinality $(3^{m-1}-1)(3^m+1), (3^{m-1})(3^m+1),$ and $(3^{m-1})(3^m+1)$ respectively.  Then,the following are negative Latin square type PDSs in $G \times G'$ with parameters $(3^{m+n-1}-1)(3^{m+n}+1), (3^{m+n-1})(3^{m+n}+1),$ and $(3^{m+n-1})(3^{m+n}+1)$ respectively: $$\widehat{C_0} = \left([\left(L_0 \cup \{ 1_G \} \right) \times \left(C_0 \cup \{ 1_{G'}\}\right)] \cup \left(L_1 \times C_1\right) \cup \left(L_2 \times C_2\right)\right) - \{ 1_G \times 1_{G'} \},$$
$$\widehat{C_1} = [(L_0 \cup \{ 1_G \} ) \times C_1] \cup [L_1 \times C_2] \cup [L_2 \times (C_0 \cup \{ 1_{G'}\})], $$
$$\widehat{C_2} = [(L_0 \cup \{ 1_G \} ) \times C_2] \cup [L_1 \times (C_0 \cup \{ 1_{G'}\})] \cup [L_2 \times C_1].$$
Similarly we have Latin square type PDSs in $G \times G$ with parameters $(3^{2m-1}+1)(3^{2m}-1), (3^{2m-1})(3^{2m}-1),$ and $(3^{2m-1})(3^{2m}-1)$ respectively: $$\widehat{L_0} = ([(L_0 \cup \{ 1_G \} ) \times (L_0 \cup \{ 1_{G})] \cup [L_1 \times L_1] \cup [L_2 \times L_2]) - \{ 1_G \times 1_{G} \},$$
$$\widehat{L_1} = [(L_0 \cup \{ 1_G \} ) \times L_1] \cup [L_1 \times L_2] \cup [L_2 \times (L_0 \cup \{ 1_{G}\})], $$
$$\widehat{L_2} = [(L_0 \cup \{ 1_G \} ) \times L_2] \cup [L_1 \times (L_0 \cup \{ 1_{G}\})] \cup [L_2 \times L_1].$$
Finally, we could instead construct Latin square type PDSs in $G' \times G'$ with parameters $(3^{2n-1}+1)(3^{2n}-1), (3^{2n-1})(3^{2n}-1),$ and $(3^{2n-1})(3^{2n}-1)$ respectively: $$\widehat{L_0} = ([(C_0 \cup \{ 1_{G'} \} ) \times (C_0 \cup \{ 1_{G'})] \cup [C_1 \times C_1] \cup [C_2 \times C_2]) - \{ 1_{G'} \times 1_{G'} \},$$
$$\widehat{L_1} = [(C_0 \cup \{ 1_{G'} \} ) \times C_1] \cup [C_1 \times C_2] \cup [C_2 \times (C_0 \cup \{ 1_{G'}\})], $$
$$\widehat{L_2} = [(C_0 \cup \{ 1_{G'} \} ) \times C_2] \cup [C_1 \times (C_0 \cup \{ 1_{G'}\})] \cup [C_2 \times C_1].$$

\end{cor}

For example, consider the case of $v = 729$. In this case, we can have a partition of the nonidentity elements of the groups ${\Z_3}^2 \times {\Z_3}^4$, ${\Z_3}^2 \times G_{1}^1$, ${\Z_3}^2 \times G_{1}^{-1}$, ${\Z_3}^2 \times G_{2}$, ${\Z_3}^2 \times \Z_{9} \times \Z_{9}$ or ${\Z_3}^2 \times \Z_{9} \rtimes_{7} \Z_{9}$ into three PDSs of cardinalities $260, 234,$ and $234$ respectively or $224, 252,$ and $252$ respectively. Moreover, by \cite[p. 1645]{Polhill_2019}, we also have a partition of $ \Z_{27} \times \Z_{27}$ into three PDSs of cardinalities $260, 234,$ and $234$, respectively.  In fact, direct calculation in \GAP \cite{GAP4} shows that $\Z_{27} \rtimes_{19} \Z_{27}$ also has such a partition into three PDSs of cardinalities $260$, $234$, and $234$, respectively.  (The PDSs in $\Z_{27} \rtimes_{19} \Z_{27}$ are obtained from those in $\Z_{27} \times \Z_{27}$ analogously as when moving from $G_t$ to $\widehat{G_t}$ in Section \ref{sec:SemidirectLargeCenter}.)  

\end{subsection}


\section{Possible next steps}
\label{sect:possnext}

While we have constructed PDSs in several nonabelian groups, we believe there is much more to be uncovered.  We list some open questions.

\begin{itemize}

\item[(1)]  All the constructions in this paper are in $p$-groups.  There have been some constructions of Latin and negative Latin square type PDSs in nonabelian non-$p$-groups, and in particular for $|G|=100$ (see \cite{Jorgensen_Klin_2003} and \cite{Smith_1995}), but aside from these and a few other small examples little is known.  It seems likely that there will be some nonabelian groups with PDSs having the same parameters as those that exist in certain abelian groups, and perhaps (such as with $|G|=100$ there might be some genuinely nonabelian parameters).

\item[(2)] We saw four distinct techniques in this paper that used abelian PDSs to obtain nonabelian PDSs: using quadratic forms and analyzing affine polar graphs, exploiting groups with a large center, calculating group ring equations in place of characters, and identifying that certain product constructions depend only on parameters. Abelian PDSs have been extensively studied, and there are many other techniques to explore from this previous work. One could consider additional ways to carry over the well-developed techniques from abelian groups to the less familiar nonabelian setting.  Especially in light of \cite{Ott_2023}, it seems likely that at least some techniques from character theory would fit this description.

\item[(3)] Theorem \ref{Productstononabelian} from Section \ref{sec:product} could be applied to other results from abelian groups. In particular, there are likely to be many nonabelian PDSs in 2-groups. (For example, consider the results of \cite{Feng_He_Chen_2020} combined with product theorems such as Theorem \ref{Productstononabelian}.).

\item[(4)] In Section \ref{sec:product}, the objective was to see that the technique of certain product constructions can carry over to nonabelian input groups with the appropriate partition. One starts to see that many nonabelian groups will have PDSs that are the same as the abelian case. One could begin to catalog all the groups that support $(v,k,\lambda, \mu)$-PDSs for relatively small $v$.

\item[(5)] Since PDSs produce strongly regular Cayley graphs, one could also begin to catalog which groups have PDSs that correspond to the various nonisomorphic $(v,k,\lambda, \mu)$-strongly regular graphs for small $v$.

\end{itemize}

\end{document}